\documentclass[11pt]{amsart}
\usepackage{amsmath, amssymb, amsthm}
\usepackage{enumitem}
\usepackage{graphicx}
\usepackage[margin=1.25in]{geometry}
\usepackage{url}
\usepackage{hyperref}

\hypersetup{
    colorlinks=true,
    linkcolor=blue,
    citecolor=blue,
    urlcolor=blue
}

\def\Z{\ensuremath{\mathbb{Z}}}

\def\C{\ensuremath{\mathbb{C}}}

\newcommand{\pa}[1]{\left(#1\right)}
\newcommand{\cpa}[1]{\left\{#1\right\}}

\newcommand{\tn}[1]{\textnormal{#1}}
\newcommand{\br}[1]{\left[#1\right]}

\newcommand{\card}[1]{\left| #1 \right|}

\newcommand{\vol}[1]{\operatorname{vol}\pa{#1}}
\newcommand{\sys}[1]{\operatorname{sys}\pa{#1}}

\newtheorem{theorem}{Theorem}
\newtheorem{lemma}{Lemma}
\newtheorem{proposition}{Proposition}

\newtheorem{corollary}{Corollary}

\title{Mazur's Knot and the Octahedron}

\author{Jack Calcut}
\address{Department of Mathematics, Oberlin College}
\email{jcalcut@oberlin.edu}

\author{Yangyang Du}
\address{Department of Mathematics, University of Michigan}
\email{yyadu@umich.edu}

\subjclass[2020]{
Primary 57K32;
Secondary 57K10, 57K31, 57R65
}

\keywords{
Mazur manifold,
Jester manifold,
contractible 4-manifold,
hyperbolic 3-manifold,
regular ideal octahedron,
thrice-punctured sphere,
systolic geodesic,
Whitehead link}

\date{\today}

\begin{document}

\begin{abstract}
Mazur's knot exterior admits a geometric description
using a single regular ideal octahedron.
The resulting hyperbolic structure is closely related to the Whitehead
link exterior through Adams' theorem on thrice-punctured spheres.
The same octahedral framework applies to the family of Jester manifolds
introduced by Sparks.
Using hyperbolic geometry, Thurston's hyperbolic Dehn filling theorem,
and Mostow--Prasad rigidity, we prove that Mazur and Jester boundary
3-manifolds are pairwise distinct up to finite ambiguity.
Using recent results on systolic geodesics, we remove the remaining finite
ambiguity and prove that the boundaries of all Mazur and Jester manifolds
are pairwise nonhomeomorphic, regardless of orientation.
Consequently, the corresponding compact, contractible $4$-manifolds are
pairwise nonhomeomorphic.
\end{abstract}

\maketitle

\section{Introduction}\label{sec:introduction}

Compact, contractible manifolds play a central role in manifold topology.
A natural question is whether every such manifold is trivial---meaning
homeomorphic to the disk $D^n$.
That holds in low dimensions: it is classical for $n \leq 2$,
and in dimension $3$ it follows from Perelman's 2003 proof of the
Poincar\'{e} conjecture.
In contrast, around 1948 Newman~\cite{newman} constructed nontrivial examples
in every dimension $n \geq 5$.
A decade later, Mazur~\cite{mazur} and Po\'{e}naru~\cite{poenaru}
independently discovered the first such examples in dimension $4$.
Those examples and their boundaries reflect the subtlety of dimension $4$.

\begin{figure}[htbp]
    \centering
    \begin{minipage}{0.4\textwidth}
        \centering
        \includegraphics[scale=0.8]{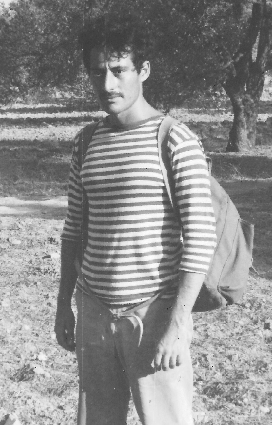}
    \end{minipage}
\hspace{0.01\textwidth}
    \begin{minipage}{0.4\textwidth}
        \centering
        \includegraphics[scale=1.1]{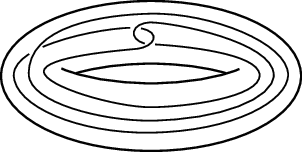}
    \end{minipage}
    \caption{Barry Mazur, Crete, 1964, photo courtesy of Gretchen Mazur,
    and Mazur's knot $K$ in $S^1\times D^2\subset S^1\times S^2$.}
    \label{fig:mazur_combined}
\end{figure}

Mazur's construction is based on the knot $K\subset S^1\times S^2$
shown in Figure~\ref{fig:mazur_combined}.
Using this knot, Mazur constructed a sequence $M_k$,
indexed by $k\in\Z$,
of smooth, compact $4$-manifolds with boundary $\partial M_k$.
Mazur showed that each $M_k$ is contractible and that
$\pi_1\pa{\partial M_{-3}}$ surjects onto the
$(2,5,7)$ hyperbolic triangle group.
Thus, $M_{-3}$ is a nontrivial compact, contractible $4$-manifold.
Two immediate questions arise:
are all of Mazur's $4$-manifolds nontrivial, and
are they pairwise nonhomeomorphic?

Laudenbach~\cite{laudenbach} answered the first question affirmatively in 1979
by showing that every $\partial M_k$ is not simply connected
using Dehn's Lemma and the Loop Theorem.
Akbulut and Kirby~\cite[p.~263]{ak} noted that Laudenbach and Eaton
independently proved that result;
however, Eaton’s proof is unknown,
and Eaton passed away in 2004.

Our focus is on the second question.
Partial progress comes from studying fundamental groups.
Using the computational algebra system MAGMA,
we verified that the groups $\pi_1 \pa{\partial M_k}$
are pairwise nonisomorphic for $\card{k} \leq 2000$
by comparing the abelianizations of their low-index subgroups
of index at most $15$.

Rather than relying on such black-box calculations, we give a geometric
description of the exterior $X$ of Mazur’s knot in $S^1\times S^2$
using a single regular ideal octahedron.
Our starting point is the observation that $X$ contains a naturally
occurring thrice-punctured sphere $\Sigma$.
Cutting $X$ along $\Sigma$ and regluing along the two copies of
$\Sigma$ yields the exterior $W$ of the Whitehead link in $S^3$.

Riley~\cite[pp.~127–128]{riley79} discovered the hyperbolic structure on $W$ in the 1970s;
see also his retrospective account~\cite{riley13} for historical context
and Wielenberg~\cite[\S~6]{wielenberg}.
Thurston~\cite[\S~3.3]{thurston80} gave an explicit geometric construction by
gluing the faces of the regular ideal octahedron;
see also Ratcliffe~\cite[\S\S~10.3,~10.6]{ratcliffe}.

Using a theorem of Adams~\cite{adams}, we transfer the hyperbolic structure
on $W$ to $X$, obtaining
\[
\vol{X}=\vol{W}=3.66386\ldots
\]
which is the volume of the regular ideal octahedron.
That gives a direct geometric proof that the
thrice-punctured sphere $\Sigma\subset X$ is totally geodesic,
that $K$ is not equivalent to $S^1\times\cpa{*}$
by any homeomorphism of $S^1\times S^2$,
and determines exactly the cusp shape of $X$.
The same geometric framework also applies to the closely related
sequence of Jester $4$-manifolds introduced by Sparks~\cite{sparks}.
This framework leads to the following theorem, which is our main result.

\begin{theorem}\label{maintheorem}
The boundaries of all Mazur and Jester $4$-manifolds are pairwise nonhomeomorphic,
regardless of orientation.
Consequently, the corresponding smooth, compact, contractible $4$-manifolds
are pairwise nonhomeomorphic.
\end{theorem}

Since Mazur's original construction, many authors have introduced
broader families of Mazur-type manifolds.
Among the earliest and most influential are the generalizations
introduced by Akbulut and Kirby~\cite{ak} and developed further
by Akbulut~\cite[Ch.~10]{ak16}, leading to corks and
numerous applications in smooth $4$-manifold topology.
In this paper, however, we return to Mazur's original sequence of compact
contractible $4$-manifolds, introduced in 1961, together with the closely
related Jester manifolds introduced by Sparks in 2018.

Our proofs are fundamentally geometric.
Theorem~\ref{finitemaintheorem} establishes Theorem~\ref{maintheorem}
up to finite ambiguity using the octahedral framework together with
Adams' theorem, Thurston's hyperbolic Dehn surgery theorem,
and Mostow--Prasad rigidity.
The remaining finite ambiguity is resolved using recent results
on systolic geodesics and explicit computations.
Both theorems answer affirmatively the two questions of
Sparks~\cite[Que.~3.5]{sparks} asking whether there exists a Jester manifold
that is not homeomorphic to the $4$-ball and whether there are infinitely
many Jester manifolds.

Throughout, all hyperbolic structures are complete and finite-volume.
A hyperbolic structure on a manifold with boundary
means a hyperbolic structure on its interior.

\section*{Acknowledgment}
The authors are grateful to Nathan Dunfield for invaluable
conversations on SnapPy.

\section{Mazur's Manifolds}\label{sec:mazurmanifolds}

We define Mazur's manifolds and provide some insight into his construction.
Mazur discovered these examples before modern handlebody 
theory had been developed.
We use Akbulut’s dotted circle notation for 1-handles, introduced
in the mid-1970s~\cite[p.~100]{ak77}.
Akbulut's notation provides a particularly effective way to describe
handle decompositions of 4-manifolds.
For further background, see Akbulut~\cite[Ch.~1]{ak16}
and Gompf and Stipsicz~\cite[Chs.~4--5]{gs}.

Begin with the solid torus $S^1\times D^3$.
One attaches a $2$-handle along a
knot $\kappa$ in the boundary $S^1\times S^2$ of $S^1\times D^3$
together with a framing.
In the corresponding surgery diagram in $S^3$,
a Seifert surface determines the $0$-framing,
and the general framing is specified by an integer $k$.
Let $V(\kappa,k)$ denote the resulting smooth, compact $4$-manifold,
and let $a\in\Z$ denote the algebraic winding number of $\kappa$
around the $S^1$-factor.
By van Kampen's theorem, the fundamental group of $V(\kappa,k)$
is isomorphic to $\Z/a\Z$.

For example, if $\kappa=S^1\times\cpa{*}$, then the handle diagram
for $V(\kappa,k)$ is shown in Figure~\ref{fig:Mazurknottorus3}.
\begin{figure}[htbp!]
    \centering
    \includegraphics[scale=1.0]{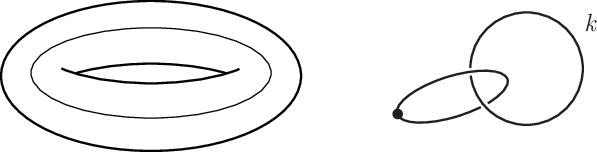}
    \caption{Knot $\kappa=S^1\times\cpa{*}$ in $S^1\times S^2$ (left)
    and handle diagram for $V(\kappa,k)$ (right).}
    \label{fig:Mazurknottorus3}
\end{figure}
Independent of the framing, the $1$- and $2$-handles cancel and
$V(\kappa,k)$ is diffeomorphic to $D^4$.
More generally, if $\kappa$ meets some $2$-sphere $\cpa{*}\times S^2$
transversely and exactly once,
then the $3$-dimensional light bulb theorem---see
Rolfsen~\cite[p.~257]{rolfsen}---shows that $\kappa$ is isotopic
to $S^1\times\cpa{*}$ and $V(\kappa,k)$ is again diffeomorphic to $D^4$.
Thus, geometric winding number one forces triviality.

For knots in $S^1\times S^2$,
it is classical that a knot is isotopic to $S^1\times\cpa{*}$ 
if and only if its exterior has fundamental group $\Z$.
The following lemma determines when $V(\kappa,k)$ is contractible.

\begin{lemma}\label{lem:contractible}
The $4$-manifold $V(\kappa,k)$ is contractible if and only if
the algebraic winding number $a$ of $\kappa$ equals $\pm1$.
In that case, $\partial V(\kappa,k)$ is an integral homology $3$-sphere.
\end{lemma}

\begin{proof}
If $V(\kappa,k)$ is contractible, then $V(\kappa,k)$ is simply connected
and $a=\pm1$.
Conversely, suppose $a=\pm1$.
We will prove the stronger fact---observed by Mazur~\cite{mazur}---that $V'=V(\kappa,k)\times[0,1]$ is diffeomorphic to $D^5$.
The $5$-manifold $V'$ is obtained from $S^1\times D^3\times[0,1]$---which
is a $5$-dimensional $0$-handle union a $5$-dimensional $1$-handle---by
adding a $5$-dimensional $2$-handle along $\kappa'=\kappa\times\cpa{0}$.
As $a=\pm1$, $\kappa'$ is homotopic to $S^1\times\cpa{*}\times\cpa{0}$
in the boundary $B$ of $S^1\times D^3\times[0,1]$.
By general position, there is a smoothly immersed cylinder $C$ in $B$
bounded by $\kappa'$ and $S^1\times\cpa{*}\times\cpa{0}$.
As $B$ is $4$-dimensional, $C$ has finitely many double-points in its interior.
Remove those double-points by finitely many finger moves until $C$ becomes embedded.
The embedded cylinder $C$ yields an ambient isotopy in $B$ carrying
$\kappa'$ to $S^1\times\cpa{*}\times\cpa{0}$.
Now, the $1$- and $2$-handles cancel to give $D^5$ as desired.

The long exact homology and cohomology sequences for the pair
$\pa{V(\kappa,k),\partial V(\kappa,k)}$
and Poincar\'{e}-Lefschetz duality imply that $\partial V(\kappa,k)$
is an integral homology $3$-sphere.
\end{proof}

Therefore, if we hope to find a $V(\kappa,k)$ that is contractible
and nontrivial, then we must use a knot $\kappa$ with
algebraic winding number equal to $\pm1$
and geometric winding number not equal to $\pm1$.
Mazur's knot $K$ is one of the simplest possible candidates.
\begin{figure}[htbp!]
    \centering
    \includegraphics[scale=1.0]{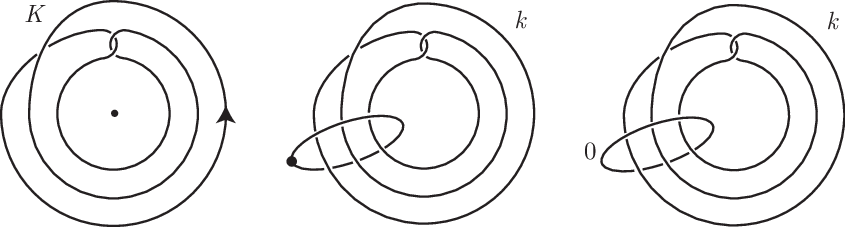}
    \caption{Mazur's knot $K\subset S^1\times S^2$ (left),
    handle diagram for Mazur's $4$-manifold $M_k$ (middle),
    and surgery diagram for the boundary $\partial M_k$ (right).
    Throughout, knots drawn around a central dot are understood to lie in
    $S^1\times D^2\subset S^1\times S^2$.
    Winding around the dot corresponds to the $S^1$-factor.}
    \label{fig:Mazur}
\end{figure}
The handle diagram for Mazur's $4$-manifold $M_k=V(K,k)$, $k\in\Z$,
and the surgery diagram for $\partial M_k$ are shown in
Figure~\ref{fig:Mazur}.

The link in Figure~\ref{fig:Mazur} (right) is symmetric:
there is an ambient isotopy interchanging its two components.
This symmetry was noted by Akbulut and Kirby~\cite[p.~282]{ak}
and illustrated explicitly by Akbulut~\cite[Fig.~4]{ak91}.
It later became the basis for the boundary involutions
used in the theory of corks~\cite[Ch.~10]{ak16}.

We have chosen an orientation of Mazur's knot
indicated in Figure~\ref{fig:Mazur} (left).
This specifies an ordered and oriented meridian--longitude homology basis
$\br{\mu_K,\lambda_K}$ of the boundary $\partial X$ of the exterior $X$
of Mazur's knot.
Namely, $\mu_K$ is the boundary of a $2$-disk that meets $K$ at one
interior point (transversely), and the linking number of $\mu_K$ and $K$ is one.
The longitude $\lambda_K$ is topologically parallel to $K$,
meets $\mu_K$ at one point (transversely),
and has linking number zero with $K$.
The writhe of $K$ is $-3$, so $\lambda_K$ is the blackboard longitude
$\lambda_{BBK}$ together with three positive twists in the $\mu_K$ direction.

By Lemma~\ref{lem:contractible}, each $M_k$ is contractible
with boundary $\partial M_k$ an integral homology $3$-sphere.
By Laudenbach~\cite{laudenbach}, each $\partial M_k$ is not
simply connected.
Therefore, each Mazur manifold $M_k$ is a nontrivial
compact, contractible $4$-manifold.

Akbulut and Kirby~\cite{ak} introduced two two-parameter families of
Mazur-type manifolds, denoted $W^{+}(l,k)$ and $W^{-}(l,k)$.
Their Proposition~1 shows that every manifold in those families is
diffeomorphic to a manifold of the form $W^{+}(l,0)$; see also
Akbulut~\cite[p.~23]{ak16}.
Consequently, their construction yields a single one-parameter family.
Using orientation reversal together with their Proposition~1, we obtain:
\[
M_k \cong -W^{-}(0,-k)
\cong W^{+}(0,k+3)
\cong W^{+}(k+3,0)
\]
where $\cong$ denotes orientation-preserving diffeomorphism.

\section{Mazur's Knot Exterior and the Octahedron}\label{sec:octahedron}

Adams~\cite{adams} proved the following remarkable theorem in the early 1980s.

\begin{theorem}[Adams]
Let $M$ be an orientable finite-volume hyperbolic $3$-manifold,
and let $S\subset M$ be an incompressible thrice-punctured sphere.
Cut $M$ along $S$ and reglue along the resulting copies of $S$
by a homeomorphism that reverses their induced boundary orientations.
Then, the resulting oriented $3$-manifold is hyperbolic
with the same volume as $M$.
\end{theorem}

Adams also proved that incompressible thrice-punctured spheres
in finite-volume hyperbolic $3$-manifolds are isotopic to
totally geodesic thrice-punctured spheres.
Our approach does not require that part of Adams' work,
since the octahedral structure will identify the totally geodesic
surface directly.
The mapping class group of the thrice-punctured sphere is
the symmetric group $S_3$---see Farb and Margalit~\cite[Prop.~2.3]{fm}.
Thus, there are only finitely many possible regluings.
A key feature of the thrice-punctured sphere is that it
admits a unique complete finite-area hyperbolic structure.

Mazur's knot exterior $X\subset S^1\times S^2$
contains the naturally occurring thrice-punctured sphere
$\Sigma=\pa{\cpa{*}\times S^2}\cap X$
depicted in Figure~\ref{fig:tps}.
\begin{figure}[htbp!]
    \centering
    \includegraphics[scale=1.0]{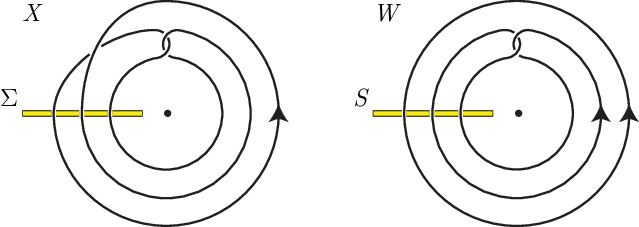}
    \caption{Cutting the knot exterior $X\subset S^1\times S^2$
    along the thrice-punctured sphere $\Sigma$, swapping the two
    leftmost punctures, and regluing yields the link exterior
    $W\subset S^1\times S^2$.}
    \label{fig:tps}
\end{figure}
Cut $X$ along $\Sigma$ yielding two copies
$\Sigma_{+}$ (upper) and $\Sigma_{-}$ (lower) of $\Sigma$,
swap the two leftmost punctures of $\Sigma_{+}$ in Figure~\ref{fig:tps}
by simply isotoping the leftmost puncture
behind the second puncture from the left,
and reglue to obtain the link exterior $W\subset S^1\times S^2$.
The exterior of the outer link component is $S^1\times D^2$.
Shrinking this exterior fiberwise in the $S^2$ factor of $S^1\times S^2$
onto a neighborhood of its core curve
yields the diagram in Figure~\ref{fig:cut1} (left),
namely the Whitehead link exterior in $S^3$.
In particular, $W$ is homeomorphic to the Whitehead link exterior.

\begin{figure}[htbp!]
    \centering
    \includegraphics[scale=1.0]{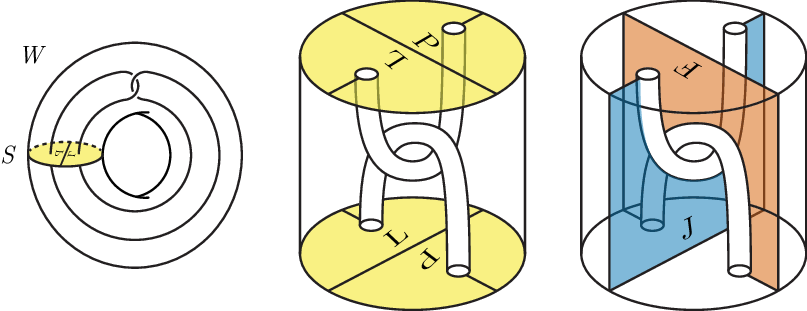}
    \caption{Whitehead link exterior $W$, genus-$2$ handlebody obtained by
    cutting $W$ along the thrice-punctured sphere $S$,
    and two $2$-disks F and J in the handlebody.}
    \label{fig:cut1}
\end{figure}

Cut $W$ along $S$ to obtain the $3$-ball with
two holes---a genus-$2$ handlebody---shown
in Figure~\ref{fig:cut1} (middle).
Cut that handlebody along the two $2$-disks F and J
shown in Figure~\ref{fig:cut1} (right).
The result is the octahedron in Figure~\ref{fig:octahedra} (left).
\begin{figure}[htbp!]
    \centering
    \includegraphics[scale=1.0]{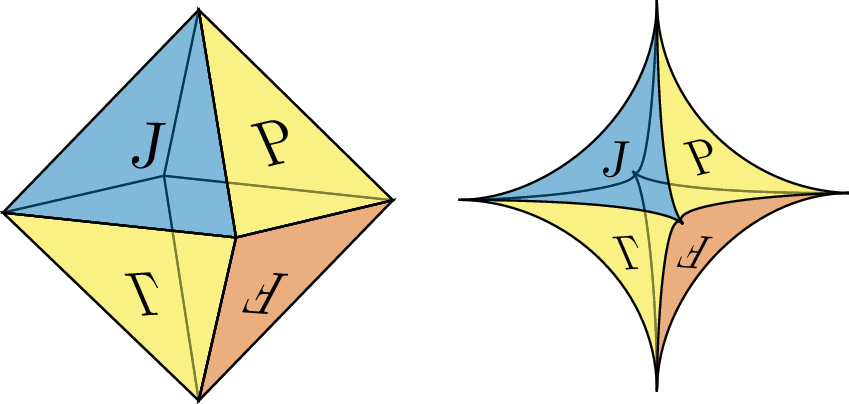}
    \caption{A regular Euclidean octahedron and the regular ideal hyperbolic octahedron.}
    \label{fig:octahedra}
\end{figure}
Reversing this cutting process recovers $W$.
Figure~\ref{fig:OctahedralNet} is the corresponding octahedral net.

\begin{figure}[htbp!]
    \centering
    \includegraphics[scale=0.73]{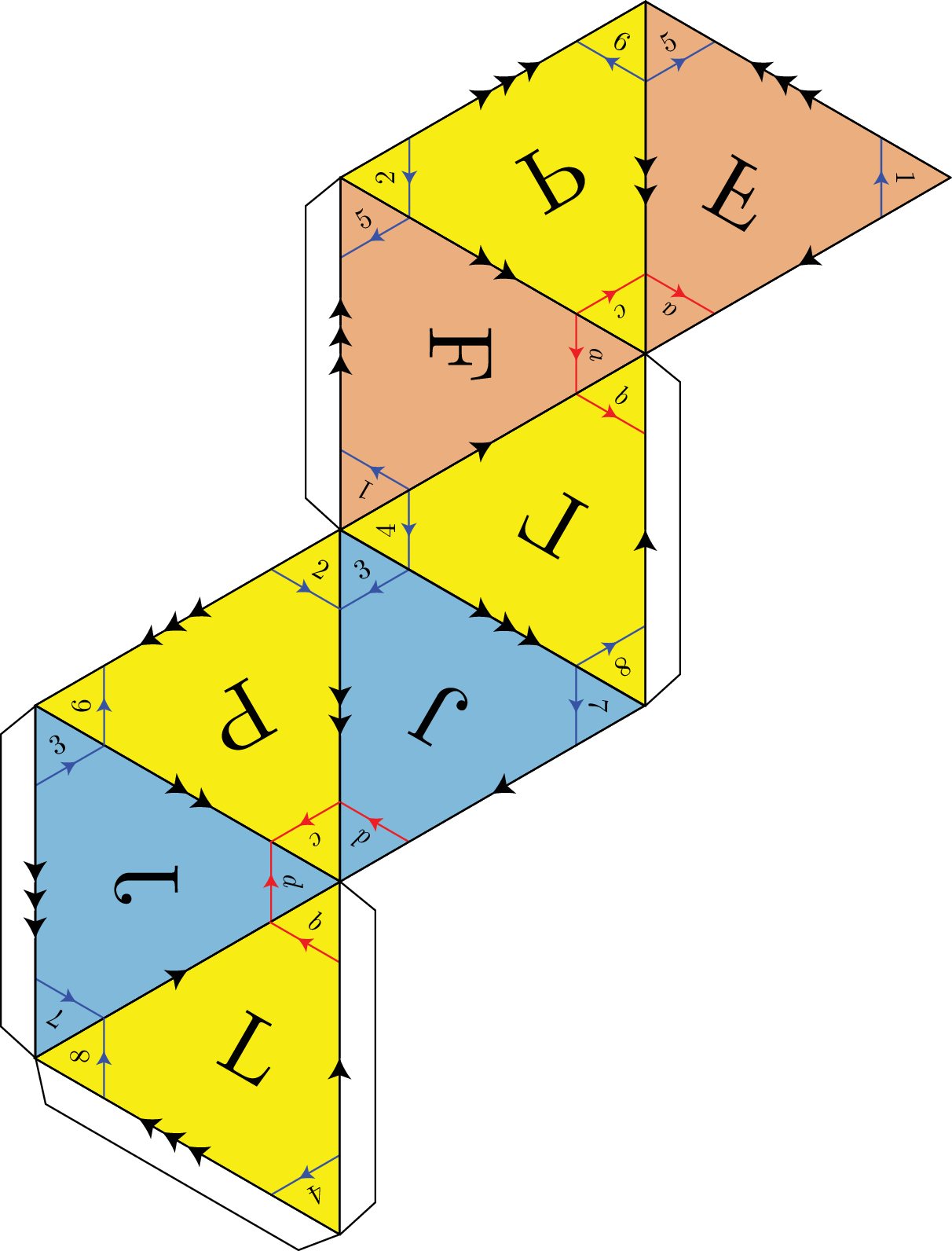}
    \caption{Octahedral net for the Whitehead link exterior $W$.
    Identifying faces with matching labels recovers $W$.
    The yellow faces P and L glue to the
    totally geodesic thrice-punctured sphere $S\subset W$.}
    \label{fig:OctahedralNet}
\end{figure}

Realize the octahedron as the regular ideal octahedron
in Poincar{\'e}'s ball model of
hyperbolic $3$-space---see Figure~\ref{fig:octahedra}.
Since the dihedral angles are $\pi/2$ and four octahedral edges
meet around each edge class in the quotient,
the total angle around each edge is $2\pi$.
Thus, the face pairings determine a hyperbolic structure on $W$.
All faces of the octahedron lie in geodesic planes,
and the faces P and L are glued straight across
the face pairings.
Hence, they descend to a totally geodesic
thrice-punctured sphere $S\subset W$.
As $S$ is totally geodesic in $W$, $S$ is incompressible in $W$.

Applying Adams' theorem now shows that $X$ is hyperbolic
and $\vol{X}=\vol{W}=3.66386\ldots$,
the volume of the regular ideal octahedron.
Since the regluing is performed along the
totally geodesic thrice-punctured sphere $S$,
the image $\Sigma\subset X$
is totally geodesic and hence incompressible in $X$.
Hyperbolicity of $X$ implies that Mazur's knot $K$
is not equivalent to $S^1\times\cpa{*}$
by any homeomorphism of $S^1\times S^2$,
since $S^1\times\cpa{*}$ has exterior $S^1\times D^2$
which is not hyperbolic.
\emph{The regular ideal octahedron completely determines
the hyperbolic geometry of Mazur's knot exterior.}

Goerner~\cite{goerner} classified the orientable octahedral
hyperbolic $3$-manifolds.
There are exactly two obtained from a single regular ideal octahedron:
the Whitehead link exterior and its sibling.
Thus, although Mazur's knot exterior has the same volume as the regular
ideal octahedron, it cannot be obtained by gluing the faces of a single
regular ideal octahedron.
Instead, our octahedral description arises by cutting along a
thrice-punctured sphere and applying Adams' theorem.
Consequently, the hyperbolic volumes of the Whitehead link exterior
and Mazur's knot exterior are determined exactly rather than numerically.

Agol~\cite[\S~3]{agol} observed that the SnapPy census manifold $m137$
is closely related to the Whitehead link exterior.
We sharpen that observation by proving that the exterior
of Mazur's knot is precisely $m137$.

Our octahedral description of the Whitehead link exterior agrees with
Thurston's original construction in his notes~\cite[\S~3.3]{thurston80}.
Thurston's later exposition~\cite[Ex.~3.3.9]{thurston97}
and Ratcliffe's description~\cite[\S~10.3]{ratcliffe}
are obtained by reflecting our entire decomposition,
including the regular ideal octahedron.
This reflects the chirality of the Whitehead link:
those constructions describe the mirror image of its exterior.

We now turn to the closely related Jester manifolds.
Their geometry is governed by the same octahedral framework,
which we will later use to classify all Mazur and Jester manifolds.

\section{Jester Manifolds}\label{sec:jester}

Sparks~\cite{sparks} introduced Jester manifolds in his construction
of $4$-dimensional splitters.
A closed splitter is an $n$-manifold that is the union of two $n$-disks
that meet in an $n$-disk.
An open splitter is defined analogously with $n$-disks
replaced by open $n$-disks.
Ancel and Guilbault~\cite{ag} proved that every compact,
contractible topological manifold of dimension $\geq5$
is a closed splitter.
Interest in open splitters was rekindled by Gabai's proof that
Whitehead's $3$-manifold is an open splitter~\cite{gabai}.
For further background on splitters, see Sparks~\cite[\S\S~1--2]{sparks}.

Ancel, Guilbault, and Sparks asked whether any Mazur
$4$-manifold $M_k$ is a closed splitter.
This intriguing question remains open.
Zeeman observed that each $M_k$ has a dunce hat spine~\cite{z64}.
Ancel and Sparks undermined one strategy for splitting an $M_k$
by showing that the dunce hat does not split into two collapsible
subpolyhedra~\cite{as}.
To avoid that difficulty, Sparks constructed Jester manifolds
with spines that decompose into two collapsible subpolyhedra.

The Jester knot construction mirrors the Mazur knot construction
through a natural two-fold symmetry.
The Jester knot $J\subset S^1\times S^2$
and the corresponding Jester $4$-manifolds $N_k=V(J,k)$,
$k\in\Z$, are shown in Figure~\ref{fig:Jester}.

\begin{figure}[htbp!]
    \centering
    \includegraphics[scale=1.0]{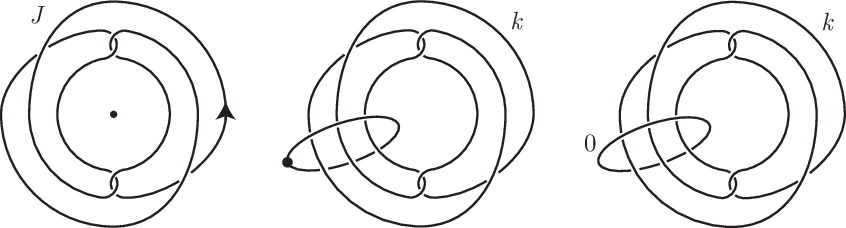}
    \caption{Jester knot $J\subset S^1\times S^2$ (left),
    handle diagram for Jester $4$-manifold $N_k$ (middle),
    and surgery diagram for the boundary $\partial N_k$ (right).}
    \label{fig:Jester}
\end{figure}

We have chosen an orientation of the Jester knot
indicated in Figure~\ref{fig:Jester} (left).
This specifies an ordered and oriented meridian--longitude homology basis
$\br{\mu_J,\lambda_J}$ of the boundary $\partial Y$ of the exterior $Y$
of the Jester knot.
The writhe of $J$ is $-6$, so $\lambda_J$ is the blackboard longitude
$\lambda_{BBJ}$ together with six positive twists in the $\mu_J$ direction.

\begin{figure}[htbp!]
    \centering
    \includegraphics[scale=1.0]{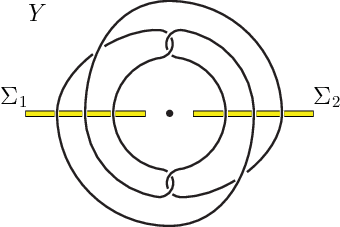}
    \caption{Jester knot exterior $Y\subset S^1\times S^2$ and
    thrice-punctured spheres $\Sigma_1$ and $\Sigma_2$.}
    \label{fig:JesterExteroior}
\end{figure}

Observe that $Y$ is obtained from two copies of $X$
by cutting them open along their copies of $\Sigma$ and
gluing together the resulting two $3$-manifolds
along their four copies of $\Sigma$.
The result is $Y$ containing two copies
$\Sigma_1$ and $\Sigma_2$ of $\Sigma$ as shown in Figure~\ref{fig:JesterExteroior}.

Adams' gluing theorem~\cite[Thm.~4.5]{adams}
shows that $Y$ is hyperbolic and
\[
\vol{Y}=2\vol{X}=7.32772\ldots
\]
which is twice the volume of the regular ideal octahedron.
Thus, the octahedral framework also determines the exact
hyperbolic volume of the Jester knot exterior.
Since the gluings are performed along totally geodesic
thrice-punctured spheres,
the images $\Sigma_1$ and $\Sigma_2$
are totally geodesic and hence incompressible in $Y$.
Hyperbolicity of $Y$ implies that Sparks' Jester knot $J$
is not equivalent to $S^1\times\cpa{*}$ by any homeomorphism of $S^1\times S^2$.

As $J$ is not equivalent to $S^1\times\cpa{*}$,
Laudenbach~\cite{laudenbach} implies that
each $\partial N_k$ is not simply connected.
Lemma~\ref{lem:contractible} implies that each $N_k$
is contractible with homology $3$-sphere boundary.
Therefore, each Jester manifold $N_k$ is
a nontrivial compact, contractible 4-manifold,
answering affirmatively a question of Sparks~\cite[Que.~3.5]{sparks}.

Sparks~\cite[p.~2137]{sparks} also introduced an alternate
Jester knot, which we denote by $J^*$---see Figure~\ref{fig:Jester1}.
We show that $J$ and $J^*$ determine the same $4$-manifolds (up to orientation)
and provide an example of two nonisotopic knots in $S^1\times S^2$
that are nevertheless equivalent by an orientation-preserving diffeomorphism.
That phenomenon does not occur for knots in the $3$-sphere.

\begin{figure}[htbp!]
    \centering
    \includegraphics[scale=0.8]{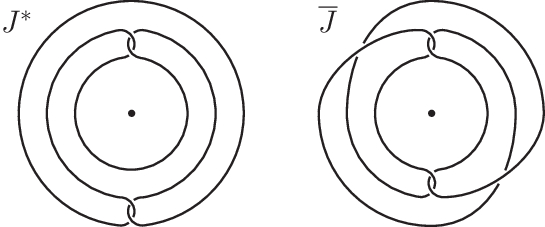}
    \caption{Equivalent but nonisotopic knots in $S^1\times S^2$.}
    \label{fig:Jester1}
\end{figure}

Let $r:S^2\to S^2$ be a reflection,
and let $\overline{J}$ denote the image of $J$ under the
orientation-reversing diffeomorphism $\tn{id}\times r$ of $S^1\times S^2$.
The following proposition shows that $J^*$ and $\overline{J}$
are equivalent but not isotopic.

\begin{proposition}\label{prop:jesterknots}
The knots $J^*$ and $\overline{J}$ in $S^1 \times S^2$ are equivalent
by an orientation-preserving diffeomorphism of $S^1 \times S^2$,
but are not isotopic.
\end{proposition}

\begin{proof}
Begin with $\overline{J}$ and perform the three isotopies in Figure~\ref{fig:Jester2}.
\begin{figure}[htbp!]
    \centering
    \includegraphics[scale=0.8]{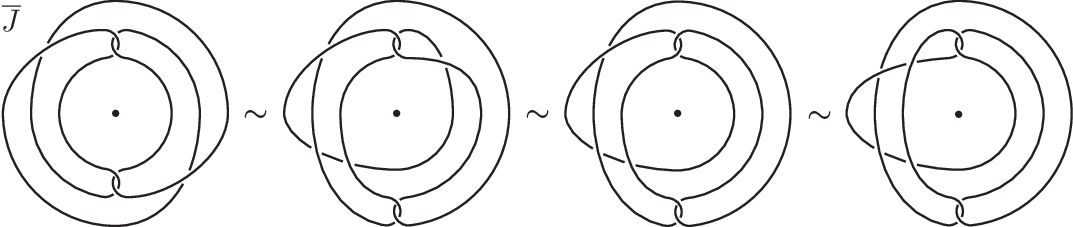}
    \caption{Three isotopies in $S^1\times S^2$.}
    \label{fig:Jester2}
\end{figure}

The resulting knot is shown in Figure~\ref{fig:Jester3} with two $2$-spheres
of the form $\cpa{*}\times S^2$ represented by arcs.
The portion of the knot between those $2$-spheres
determines the three-strand braid shown in Figure~\ref{fig:Jester3}.
\begin{figure}[htbp!]
    \centering
    \includegraphics[scale=0.8]{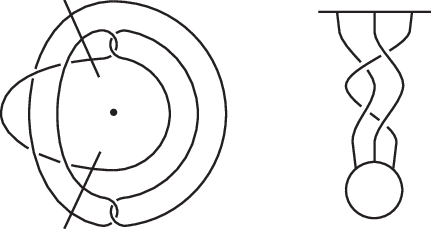}
    \caption{Knot in $S^1\times S^2$ with two $2$-spheres represented by arcs (left)
    and three-strand braid between those two $2$-spheres.}
    \label{fig:Jester3}
\end{figure}

Now perform the three isotopies of $S^1\times S^2$ shown schematically
in Figure~\ref{fig:Jester4}, where only the braid portion is illustrated.
The first and third are light bulb isotopies in which an arc swings under the lower sphere.
Finally, perform a Gluck twist.
Namely, cut $S^1\times S^2$ along the lower sphere in Figure~\ref{fig:Jester4},
rotate one boundary sphere by a complete revolution, and reglue.
The resulting braid is trivial.
\begin{figure}[htbp!]
    \centering
    \includegraphics[scale=0.8]{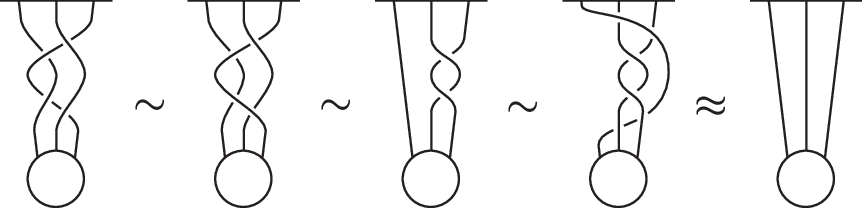}
    \caption{Three isotopies followed by a Gluck twist.}
    \label{fig:Jester4}
\end{figure}

The composition of those isotopies and the Gluck twist is an orientation-preserving diffeomorphism of $S^1\times S^2$ carrying $\overline{J}$ to $J^*$, as desired.

Finally, suppose, by way of contradiction,
that $J^*$ and $\overline{J}$ are isotopic.
Aceto, Bregman, Davis, Park, and Ray~\cite[Thm.~1.5]{aceto}
proved that a knot in $S^1\times S^2$ with odd algebraic winding number
that is isotopic to its image under the Gluck twist must have
geometric winding number $\pm1$.
By the $3$-dimensional light bulb theorem, that implies
$\overline{J}$ is isotopic to $S^1\times\cpa{*}$ with nonhyperbolic exterior.
However, the exterior of $\overline{J}$ is homeomorphic to $Y$, which is hyperbolic.
That contradiction completes the proof.
\end{proof}

Proposition~\ref{prop:jesterknots} provides an orientation-preserving 
diffeomorphism $S^1\times S^2\to S^1\times S^2$ carrying
$J^*$ to $\overline{J}$.
Composing with $\tn{id}\times r$, we obtain an orientation-reversing
diffeomorphism $f:S^1\times S^2\to S^1\times S^2$ carrying
$J^{\ast}$ to $J$.
By construction, $f$ extends to an orientation-reversing
diffeomorphism $F:S^1\times D^3\to S^1\times D^3$.
In fact, every diffeomorphism of $S^1\times S^2$ extends to
$S^1\times D^3$, but we do not need that result.
It follows that $N_k=V(J,k)$ is orientation-reversing diffeomorphic
to $V\pa{J^*,-k}$.
Therefore, the knots $J$ and $J^*$ determine the same
$4$-manifolds up to orientation.

\section{Hyperbolic Dehn Filling}\label{sec:hypdehnfill}

Let $M$ be a compact, orientable $3$-manifold whose boundary is a torus.
Fix an ordered homology basis $\br{\mu,\lambda}$ for $\partial M$.
A slope on $\partial M$ is the isotopy class of an essential simple
closed curve.
With respect to the basis $\br{\mu,\lambda}$, every slope is represented by
a simple closed curve whose homology class is $p\mu+q\lambda$, where
$p,q\in\Z$ with $\gcd(p,q)=1$, unique up to replacing $(p,q)$ by $(-p,-q)$.
We denote this slope by $(p,q)$ or, equivalently, by $p/q$.

The Dehn filled manifold $M(p,q)$ is obtained by gluing a solid torus
$S^1\times D^2$ to $M$ along their boundaries
so that the boundary $\cpa{*}\times\partial D^2$
of a meridional disk is glued to a simple closed curve
representing the slope $(p,q)$.
We refer to the glued-in solid torus as the filling.

In case $M$ is hyperbolic, a slope $(p,q)$ is exceptional provided
the filled manifold $M(p,q)$ fails to be hyperbolic.

Thurston's hyperbolic Dehn surgery theorem is a cornerstone of
$3$-manifold topology.
We record the one-cusped case---see \cite[Thm.~5.8.2]{thurston80}
and~\cite{thurston82}.

\begin{theorem}[Thurston]
Let $M$ be an orientable finite-volume hyperbolic $3$-manifold
with one cusp.
Then, $M$ has at most finitely many exceptional slopes.
Moreover,
\[
\vol{M(p,q)} < \vol{M}
\]
for every hyperbolic filled manifold $M(p,q)$, and
\[
\vol{M(p,q)} \to \vol{M} \quad \tn{as} \quad \card{p}+\card{q}\to\infty.
\]
\end{theorem}

The octahedral framework together with Adams' theorems, Thurston's
hyperbolic Dehn surgery theorem, and Mostow--Prasad rigidity
imply the following theorem,
which we regard as Theorem~\ref{maintheorem} up to finite ambiguity.

\begin{theorem}\label{finitemaintheorem}
Each Mazur manifold boundary $\partial M_k$ is homeomorphic
to at most finitely many others.
Each Jester manifold boundary $\partial N_k$ is homeomorphic
to at most finitely many others.
At most finitely many Mazur manifold boundaries are homeomorphic
to any Jester manifold boundary.
\end{theorem}

\begin{proof}
The manifold $\partial M_k$ is the Dehn filled manifold $X(k,1)$.
By Thurston's theorem, only finitely many slopes on $\partial X$ are exceptional.
For the remaining slopes, the filled manifolds $\partial M_k$ are hyperbolic
with volume less than $\vol{X}$,
and satisfy $\vol{\partial M_k}\to\vol{X}$ as $|k|\to\infty$.
Hence, each volume value occurs only finitely many times.
Mostow--Prasad rigidity now implies that each Mazur manifold boundary
is homeomorphic to at most finitely many others.

The same argument applied to $\partial N_k=Y(k,1)$
shows that each Jester manifold boundary is homeomorphic
to at most finitely many others.

Finally, the octahedral framework and Adams' theorems imply that
$\vol{Y}=2\vol{X}$.
By Thurston's theorem, $\vol{\partial N_k}\to\vol{Y}$ as $|k|\to\infty$.
Hence, all but finitely many Jester manifold boundaries
have volume greater than $\vol{X}$,
whereas every hyperbolic Mazur manifold boundary has volume
less than $\vol{X}$.
By Mostow--Prasad rigidity, at most finitely many
Mazur manifold boundaries are homeomorphic to any Jester manifold boundary.
\end{proof}

The proof of Theorem~\ref{finitemaintheorem} is entirely independent
of computer calculations.
In Section~\ref{proofofTheorem1}, we prove Theorem~\ref{maintheorem}
in its entirety using computations with various software
together with recent geometric results.
That requires the cusp shapes of $X$ and $Y$,
which we compute in the next section.

\section{Cusp Shape}\label{sec:cuspshape}

We compute the cusp shapes of the Whitehead link exterior $W$
and the Mazur and Jester knot exteriors $X$ and $Y$
directly from the octahedral framework.
Each cusp torus inherits a Euclidean structure,
unique up to similarity.
Identifying the universal cover of a Euclidean cusp torus
with $\C$, lift the meridian $\mu$
and the longitude $\lambda$ to oriented arcs beginning at $0$.
Their endpoints determine complex numbers,
which, by an abuse of notation,
we again denote by $\mu$ and $\lambda$.
The cusp shape is then
\[
\tau=\frac{\lambda}{\mu}\in\C
\]
which determines the similarity type
of the Euclidean cusp torus.

Throughout, we use the standard meridian--longitude homology bases
$\br{\mu_K,\lambda_K}$ and $\br{\mu_J,\lambda_J}$
introduced in Sections~\ref{sec:mazurmanifolds} and~\ref{sec:jester}.
We have chosen orientations in Figure~\ref{fig:tps} of the two
link components yielding $W$.
In Figure~\ref{fig:cusp1}, the corresponding cusp tori are shown
in red and blue. 
The inner (clasped) component in Figure~\ref{fig:tps} corresponds
to the red cusp torus,
and the outer component corresponds to the blue cusp torus.
This specifies ordered and oriented meridian--longitude homology bases
$\br{\mu_R,\lambda_R}$ and $\br{\mu_B,\lambda_B}$ of the respective cusp tori.
Let $\lambda_{BBR}$ and $\lambda_{BBB}$ denote the corresponding blackboard
longitudes (the first two B's in a subscript of $\lambda$ indicate blackboard).
The writhes of the inner and outer link components are $-2$ and $0$, respectively.
Thus, $\lambda_R=\lambda_{BBR}+2\mu_R$ and $\lambda_B=\lambda_{BBB}$.

In the octahedral framework,
these Euclidean structures are determined
by intersecting the regular ideal octahedron
with small horospheres about its ideal vertices.
The octahedral net in Figure~\ref{fig:OctahedralNet}
contains edge labels for the six resulting Euclidean squares,
four blue and two red.

We first compute the cusp shapes
of the two cusps of the Whitehead link exterior.
We then transfer those cusp tori backwards
along our cutting process in Section~\ref{sec:octahedron}
to compute the cusp shape of $X$.
Finally, we compute the cusp shape of $Y$
using the natural double covering $Y\to X$.

\begin{proposition}
The cusp shape of the Mazur knot exterior $X$ is $\tau_X=\frac{13}{4}+\frac{9}{4}i$.
\end{proposition}

\begin{proof}
We begin with the cusp shapes for $W$.
Our starting point is Figure~\ref{fig:cusp1} (left),
which is Figure~\ref{fig:cut1} (middle)
with cusp edges indicated.

\begin{figure}[htbp!]
    \centering
    \includegraphics[scale=1.0]{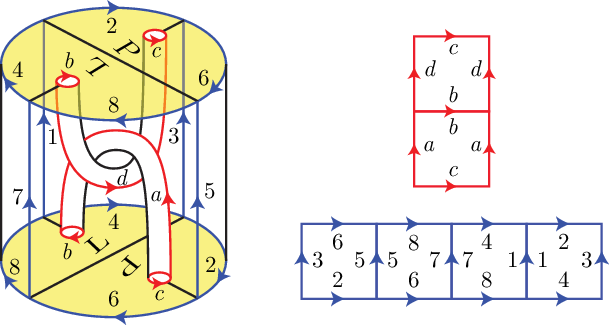}
    \caption{Genus-$2$ handlebody that glues to the Whitehead link exterior $W$ (left).
    Cusp tori of $W$ (right).}
    \label{fig:cusp1}
\end{figure}

For the red cusp, Figure~\ref{fig:cusp1} (left) shows that
$\mu_R$ corresponds to $c$
and $\lambda_{BBR}$ corresponds to traversing edges
$a$ and then $d$.
Figure~\ref{fig:cusp1} (right) now shows that
$\mu_R=1$ and $\lambda_{BBR}=2i$.
Hence, $\lambda_R=\lambda_{BBR}+2\mu_R=2+2i$.
Therefore, $\tau_R={\lambda_R}/{\mu_R}=2+2i$.

For the blue cusp, Figure~\ref{fig:cusp1} (left) shows that
$\mu_B$ corresponds to traversing edges
$6$, $8$, $4$, and then $2$,
and $\lambda_{BBB}$ corresponds to traversing edge
$5$ and then the reverse of $6$ to maintain the blackboard framing.
Figure~\ref{fig:cusp1} (right) now shows that
$\mu_B=4$ and $\lambda_B=\lambda_{BBB}=-1+i$.
Therefore, $\tau_B={\lambda_B}/{\mu_B}=-\frac{1}{4}+\frac{1}{4}i$.

We pause to compare the cusp shapes $\tau_R$ and $\tau_B$ to the
cusp shapes of the exterior of the standard Whitehead link in $S^3$
shown in Figure~\ref{fig:wl}.
\begin{figure}[htbp!]
    \centering
    \includegraphics[scale=1.0]{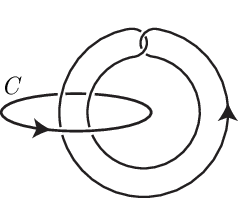}
    \caption{Whitehead link in $S^3$.}
    \label{fig:wl}
\end{figure}
The link orientation in Figure~\ref{fig:wl} specifies an ordered and
oriented meridian--longitude homology basis of each cusp torus.
For those bases, both cusp shapes equal $\tau=2+2i$,
as verified with SnapPy.
We realized $W$ as the exterior of a link in $S^1\times S^2$
and computed cusp shapes using the naturally induced peripheral bases.
There is a canonical homeomorphism $h$ from $W$ in Figure~\ref{fig:cut1} (left)
to the exterior of the Whitehead link in $S^3$ in Figure~\ref{fig:wl}.
It is obtained by noting that $W$ is exactly the exterior of the link
component $C$ in $S^3$.
The homeomorphism $h$ satisfies $h_{*}\pa{\mu_B}=\lambda_C$ and
$h_{*}\pa{\lambda_B}=-\mu_C$.
Therefore:
\[
\tau_C=\frac{\lambda_C}{\mu_C}=-\frac{\mu_B}{\lambda_B}=-\tau_B^{-1}
=-\pa{-\frac{1}{4}+\frac{1}{4}i}^{-1}=2+2i
\]
in agreement with the standard description of the Whitehead link exterior in $S^3$.

We now use these cusp tori to compute the cusp shape of $X$.
Recall the thrice-punctured sphere $S$ in Figure~\ref{fig:cut1} (left).
Cut $W$ along $S$ and let $S_{+}$ (top) and $S_{-}$ (bottom)
be the resulting copies of $S$.

\begin{figure}[htbp!]
    \centering
    \includegraphics[scale=1.0]{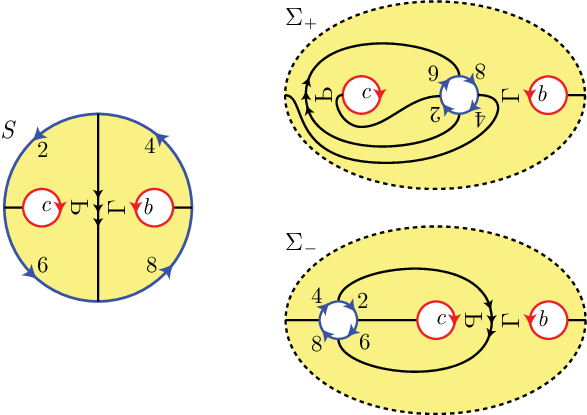}
    \caption{Thrice-punctured sphere $S$ in $W$
    and two copies $\Sigma_{+}$ and $\Sigma_{-}$ of $\Sigma$
    obtained by cutting $X$ along $\Sigma$.}
    \label{fig:cusp2}
\end{figure}

Figure~\ref{fig:cusp2} (left) shows $S$ viewed from above.
Reverse the fiberwise shrinking 
in spheres $\cpa{*}\times S^2$ described in Section~\ref{sec:octahedron}.
The rightmost point of $S$ swings around to the left on the sphere
$\cpa{*}\times S^2$ containing $S$.
Then, isotop $S_{+}$ to swing the blue circle in front of the red 
circle $c$ as in Figure~\ref{fig:cusp2}.
That yields the two thrice-punctured spheres $\Sigma_{+}$ and $\Sigma_{-}$.
We glue them together in the canonical way straight down to obtain $X$.

\begin{figure}[htbp!]
    \centering
    \includegraphics[scale=0.80]{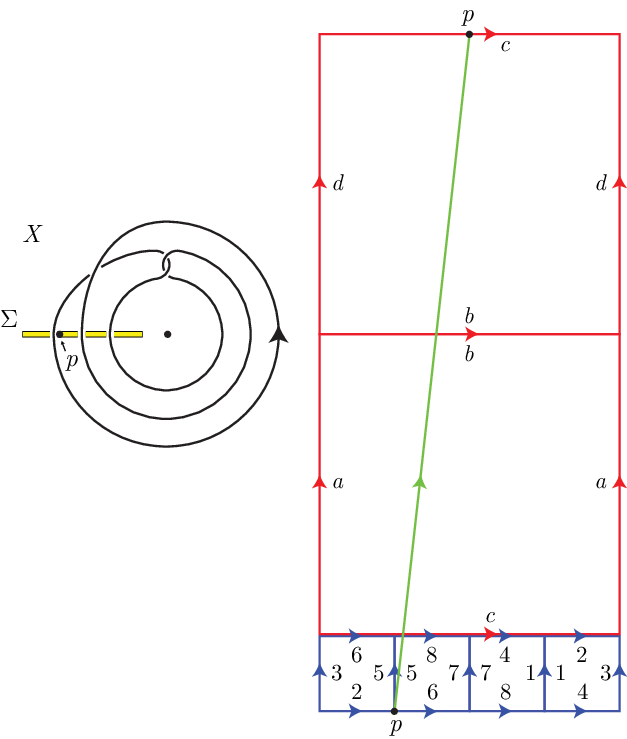}
    \caption{Mazur knot exterior $X$, thrice-punctured sphere $\Sigma$,
    and basepoint $p$ for the blackboard longitude of $K$.
    On the right is the universal cover of the Euclidean cusp torus of $X$
    with a lift of the blackboard longitude shown in green.}
    \label{fig:cusp3}
\end{figure}

Figure~\ref{fig:cusp2} shows that $\mu_K$ corresponds to traversing
edges $6$, $8$, $4$, and then $2$. 
Figure~\ref{fig:cusp3} (right) now yields $\mu_K=4$.
Let the blackboard longitude $\lambda_{BBK}$ of $K$ begin at the point $p$ where
edges $2$ and $6$ meet in $\Sigma_{-}$ as shown in Figure~\ref{fig:cusp3}.
Traverse edge $5$ and then edge $8$ to maintain the blackboard longitude.
Then, traverse a parallel to $a$ and a parallel to $d$.
The lift of the blackboard longitude is shown in Figure~\ref{fig:cusp3} (right)
and yields $\lambda_{BBK}=1+9i$.
Hence, $\lambda_K=\lambda_{BBK}+3\mu_K=13+9i$.
Therefore, $\tau_K={\lambda_K}/{\mu_K}=\frac{13}{4}+\frac{9}{4}i$.
\end{proof}

\begin{proposition}
Let $q:Y\to X$ denote the natural double cover.
Then, $q_{*}\pa{\mu_{J}}=\mu_{K}$ and $q_{*}\pa{\lambda_{J}}=2\lambda_{K}$.
Consequently, $\tau_{Y}=2\tau_{X}$.
\end{proposition}

\begin{proof}
The natural double covering $q:Y\to X$ restricts to a double cover
of cusp tori that is a local Euclidean isometry.
Note that $q_{*}\pa{\mu_J}=\mu_{K}$.
Likewise, the blackboard longitude of $J$ projects
with degree two onto the blackboard
longitude of $K$, so $q_{*}\pa{\lambda_{BBJ}}=2\lambda_{BBK}$.
Since the writhes of $K$ and $J$ are $-3$ and $-6$ respectively,
we have $\lambda_{K}=\lambda_{BBK}+3\mu_{K}$ and
$\lambda_{J}=\lambda_{BBJ}+6\mu_{J}$.
Normalizing the Euclidean cusp tori so that each meridian has length one,
we obtain $\tau_{Y}=2\tau_{X}$.
\end{proof}

\begin{corollary}
The cusp shape of the Jester knot exterior $Y$ is $\tau_Y=\frac{13}{2}+\frac{9}{2}i$.
\end{corollary}

\section{Proof of Theorem~\ref{maintheorem}}
\label{proofofTheorem1}

\noindent\textbf{Code and Data Availability.} 
The computer calculations used in Sections 6--8 were executed
using the open-source software
SageMath 10.7, SnapPy 3.3.2, and GAP 4.14.0.
To ensure reproducibility, the complete code, output, and analysis,
together with the companion paper
\emph{Computations for Mazur's Knot and the Octahedron},
are  publicly available in our dedicated GitHub repository:
\newline
\url{https://github.com/yyadu/Mazur-manifold-computation}.
\newline
We summarize the computations needed for the proof below.
Complete computational details appear in the companion paper.

SnapPy identifies the exteriors $X$ and $Y$ with the
census manifolds $m137$ and $t12072$, respectively.
Under these identifications, there are meridian--longitude
basis changes giving homeomorphisms
\[
X(k,1)\approx m137(k+3,1) \quad \tn{and} \quad Y(k,1)\approx t12072(k+6,1)
\quad \tn{for } k\in\Z.
\]

Dunfield~\cite{dunfield} determined the exceptional slopes for
the SnapPy census manifolds.
The integral exceptional slopes of $X$ are $\cpa{-5,-4,-3,-2}$,
and $Y$ has none.
Using the correspondence
\[
\partial M_k=X(k,1)\approx m137(k+3,1)
\]
Dunfield's census identifies the exceptional Mazur manifold
boundaries listed in Table~\ref{tab:exceptional}.
\begin{table}[ht]
\centering
\begin{tabular}{c|c}
$k$ & $X(k,1)$ \\
\hline
$-5$ & graph manifold \\
$-4$ & $\Sigma(3,4,5)$ \\
$-3$ & $\Sigma(2,5,7)$ \\
$-2$ & $\Sigma(2,3,13)$
\end{tabular}
\caption{The exceptional Mazur manifold boundaries $\partial M_k=X(k,1)$.}
\label{tab:exceptional}
\end{table}
The three Brieskorn sphere identifications agree with those obtained earlier by
Akbulut and Kirby~\cite{ak}---see also Akbulut~\cite[p.~161]{ak16}---via the correspondence $M_k\cong W^{+}(k+3,0)$ discussed in
Section~\ref{sec:mazurmanifolds} above.
The remaining exceptional boundary $X(-5,1)$ is identified by Dunfield
as the graph manifold
\[
SFS[D:(2,1),(3,1)]\cup_m SFS[D:(2,1),(3,2)]
\quad \tn{where} \quad
m=
\begin{bmatrix}
0&1\\
1&0
\end{bmatrix}.
\]

Historically, $X(-3,1)=\partial M_{-3}\approx\Sigma(2,5,7)$ is precisely the
boundary studied by Mazur in his original paper~\cite{mazur}.
He proved that $\pi_1(\partial M_{-3})$ surjects onto the $(2,5,7)$
hyperbolic triangle group, thereby showing that $M_{-3}$ is not
homeomorphic to the $4$-ball.
Akbulut~\cite{ak} later recognized that $\partial M_{-3}$ is the Brieskorn sphere
$\Sigma(2,5,7)$, an insight that arose from comparing presentations of its
fundamental group with presentations arising in the resolution of
complex surface singularities (personal communication).

The exceptional manifolds $X(k,1)$, $k\in\cpa{-5,-4,-3,-2}$,
are distinguished topologically from the hyperbolic
$X(k,1)$ and $Y(k,1)$ by hyperbolicity.
They are distinguished from one another by counting classes
of epimorphisms from their fundamental groups to
alternating groups $A_n$.
Here, two epimorphisms $f_1$ and $f_2$ are equivalent
provided $\phi\circ f_1=f_2$ for some automorphism $\phi$ of $A_n$.
Table~\ref{tab:countepimorphisms} collects the results of this computation.
Consequently, the relevant fundamental groups $\pi_1(X(k,1))$
are pairwise nonisomorphic.
Therefore, the exceptional filled manifolds $X(k,1)$
are pairwise nonhomeomorphic.
Thus, we exclude them from further consideration.

\begin{table}[htb]
\centering
\begin{tabular}{c|ccc}
& $A_5$ & $A_6$ & $A_7$ \\
\hline
$X(-5,1)$ & 1 & 0 & 0 \\
$X(-4,1)$ & 1 & 2 & 5 \\
$X(-3,1)$ & 0 & 0 & 2 \\
$X(-2,1)$ & 0 & 0 & 0
\end{tabular}
\caption{Numbers of classes of epimorphisms $\pi_1(X(k,1)) \twoheadrightarrow A_n$.}
\label{tab:countepimorphisms}
\end{table}

A systolic geodesic in a hyperbolic $3$-manifold
is a shortest closed geodesic.
A powerful technique is to identify the unique
systolic geodesic when it exists,
drill it, and compare the resulting drilled manifolds.

Futer, Purcell, and Schleimer~\cite{fps} proved effective
bounds guaranteeing that the core of the filling is the unique systolic
geodesic for sufficiently long slopes.
Their results utilize the systole and cusp shape of the unfilled manifold.
The verified systoles of $X$ and $Y$ are
\begin{alignat*}{2}
\sys{X}&=\sys{m137} &&=0.50979\ldots \\
\sys{Y}&=\sys{t12072} &&=1.01958\ldots
\end{alignat*}
Since both systoles exceed $0.145$, the hypotheses of
Theorems 3.19 and 3.23 of~\cite{fps}
reduce to requiring only that the normalized slope length be at least $9.97$.
The cusp shapes
\[
\tau_X=\frac{13}{4}+\frac{9}{4}i \quad\tn{and}\quad
\tau_Y=\frac{13}{2}+\frac{9}{2}i
\]
were computed geometrically in Section~\ref{sec:cuspshape}
and independently verified by SnapPy.
Substituting these cusp shapes into the normalized-length
formula yields the symmetric sufficient bounds
$\card{k}\ge19$ for $X$ and $\card{k}\ge28$ for $Y$.
Therefore, by \cite[Thm.~3.23]{fps}, the core of the filling is the
unique systolic geodesic in $X(k,1)$ for $\card{k}\ge19$
and in $Y(k,1)$ for $\card{k}\ge28$.
In those cases, drilling the unique systolic geodesic
recovers $X$ and $Y$ respectively.

The preceding discussion together with the computations
in our companion paper establish these facts.

\begin{itemize}
\item The following hold for $X(k,1)$:
\begin{itemize}[label=$\circ$]
\item For $k\notin\cpa{-7,-6,\ldots,0}$, the core of the filling is the unique
systolic geodesic.
\item For $k\in\cpa{-7,-6,-1,0}$, the systolic geodesic
is not the core of the filling but is unique.
\item For $k\in\cpa{-5,-4,-3,-2}$, $X(k,1)$ is not hyperbolic (exceptional slopes).
\end{itemize}
\item The following hold for $Y(k,1)$:
\begin{itemize}[label=$\circ$]
\item For $k\notin\cpa{-9,-8,\ldots,-3}$, the core of the filling is the unique
systolic geodesic.
\item For $k\in\cpa{-9,-8,\ldots,-4}$, the systolic geodesic
is not the core of the filling but is unique.
\item For $k=-3$, there are exactly two systolic geodesics,
the core of the filling is one of them,
and a self-isometry of $Y(-3,1)$ interchanges them.
\end{itemize}
\item Drilling the unique systolic geodesic from each
hyperbolic filled manifold yields the census manifolds
in Table~\ref{tab:systoledrilling}.
The 12 drilled manifolds listed are pairwise nonisometric.
By Mostow--Prasad rigidity, they are pairwise nonhomeomorphic.

\begin{table}[h!t]
\centering

\begin{minipage}[t]{0.45\textwidth}
\centering
\textbf{$X(k,1)$}

\smallskip
\hrule
\smallskip

\begin{tabular}{c|c}
$k$ & Systole-drilled manifold\\
\hline
$-7$ & $m081$\\
$-6$ & $m011$\\
$-5,\ldots,-2$ & exceptional slopes\\
$-1$ & $m004$\\
$0$ & $m038$\\
otherwise & $m137$
\end{tabular}
\end{minipage}
\hspace{0.25in}
\begin{minipage}[t]{0.45\textwidth}
\centering
\textbf{$Y(k,1)$}

\smallskip
\hrule
\smallskip

\begin{tabular}{c|c}
$k$ & Systole-drilled manifold\\
\hline
$-9$ & $v3519$\\
$-8$ & $v3200$\\
$-7$ & $v2919$\\
$-6$ & $v2911$\\
$-5$ & $v3184$\\
$-4$ & $v3505$\\
$-3$ & not unique\\
otherwise & $t12072$
\end{tabular}
\end{minipage}

\caption{Manifolds obtained by drilling the unique systolic geodesic
in $X(k,1)$ and $Y(k,1)$.}
\label{tab:systoledrilling}
\end{table}

\item Drilling each of the two systolic geodesics of $Y(-3,1)$
yields $t12072$.
\item The one-cusped manifolds $X$ and $Y$ have isometry groups
of orders $2$ and $4$, respectively, and every self-isometry
acts trivially on slopes.
\end{itemize}

Now, suppose two of the hyperbolic manifolds $Z$ and $Z'$ in the collection of
all $X(k,1)$ and $Y(k,1)$ are homeomorphic.
By Mostow--Prasad rigidity, $Z$ and $Z'$ are isometric.
The manifold $Y(-3,1)$ is distinguished as the unique manifold
with more than one systolic geodesic.
So, we may exclude it from further consideration.
Thus, both $Z$ and $Z'$ have a unique systolic geodesic.
The isometry $Z\to Z'$ must send one of those geodesics to the other.
Drilling $Z$ and $Z'$ yields homeomorphic systole-drilled manifolds.
Each systole-drilled manifold besides $m137$ and $t12072$
appears for exactly one parent manifold $X(k,1)$ or $Y(k,1)$.
Therefore, either both $Z$ and $Z'$ belong to the family $X(k,1)$ where $k\notin\cpa{-7,-6,\ldots,0}$ or both belong to the family $Y(k,1)$ where $k\notin\cpa{-9,-8,\ldots,-3}$.

Suppose the Mazur case occurs.
Then, we have an isometry $f:X(k,1)\to X(l,1)$
for some $k,l\notin\cpa{-7,-6,\ldots,0}$,
where one unique systolic geodesic is sent to the other
and both are cores of the fillings.
Systole-drilling yields a homeomorphism $\left.f\right|:X\to X$.
By Mostow--Prasad rigidity, $\left.f\right|$
is homotopic to a self-isometry of $X$.
Therefore, $\left.f\right|$ acts trivially on boundary slopes
since every self-isometry of $X$ acts trivially on boundary slopes.
On the other hand, the homeomorphism $f$ carries a meridional filling
$2$-disk in $X(k,1)$ to a meridional filling $2$-disk in $X(l,1)$.
Hence, the restriction $\left.f\right|:\partial X\to\partial X$
carries the unique slope in $\partial X$ bounding a
meridional filling $2$-disk for $X(k,1)$, namely $(k,1)$, to the
unique slope in $\partial X$ bounding a meridional filling
$2$-disk for $X(l,1)$, namely $(l,1)$.
Hence, $(k,1)=(l,1)$ and $k=l$ as desired.
Thus, the hyperbolic Mazur manifold boundaries are pairwise
nonisometric and nonhomeomorphic.

The proof for the Jester case is nearly identical.
That completes our proof of Theorem~\ref{maintheorem}.

\section{Conclusion}

A single regular ideal octahedron governs the geometry
of the Mazur knot exterior, the Jester knot exterior,
and the classification of the resulting
Mazur and Jester manifolds.
Our philosophy throughout has been to prove as much
as possible using geometry.
The geometry developed here encodes substantially more
information than is needed for Theorem~\ref{maintheorem}.
Representative volumes of Mazur and Jester manifold boundaries
are listed in Table~\ref{tab:volumes}.

\begin{table}[ht]
\centering

\begin{minipage}{0.46\textwidth}
\centering
\begin{tabular}{r|r}
$k$ & $\vol{\partial M_k}$ \\
\hline
$-1000$ & $3.66384\ldots$ \\
$-9$    & $3.08386\ldots$ \\
$-8$    & $2.86563\ldots$ \\
$-7$    & $2.51622\ldots$ \\
$-6$    & $1.91221\ldots$ \\
$-5$    & exceptional \\
$-4$    & exceptional \\
$-3$    & exceptional \\
$-2$    & exceptional \\
$-1$    & $1.39850\ldots$ \\
$0$     & $2.25976\ldots$ \\
$1$     & $2.71245\ldots$ \\
$2$     & $2.98683\ldots$ \\
$1000$  & $3.66384\ldots$
\end{tabular}
\end{minipage}
\hfill
\begin{minipage}{0.46\textwidth}
\centering
\begin{tabular}{r|r@{\qquad}r|r}
$k$ & $\vol{\partial N_k}$ &
$k$ & $\vol{\partial N_k}$ \\
\hline
$-1000$ & $7.32767\ldots$ &
$987$   & $7.32767\ldots$ \\
$-19$   & $7.07621\ldots$ &
$6$     & $7.07621\ldots$ \\
$-18$   & $7.03660\ldots$ &
$5$     & $7.03660\ldots$ \\
$-17$   & $6.98747\ldots$ &
$4$     & $6.98747\ldots$ \\
$-16$   & $6.92577\ldots$ &
$3$     & $6.92577\ldots$ \\
$-15$   & $6.84735\ldots$ &
$2$     & $6.84735\ldots$ \\
$-14$   & $6.74648\ldots$ &
$1$     & $6.74648\ldots$ \\
$-13$   & $6.61541\ldots$ &
$0$     & $6.61541\ldots$ \\
$-12$   & $6.44415\ldots$ &
$-1$    & $6.44415\ldots$ \\
$-11$   & $6.22139\ldots$ &
$-2$    & $6.22139\ldots$ \\
$-10$   & $5.94064\ldots$ &
$-3$    & $5.94064\ldots$ \\
$-9$    & $5.62503\ldots$ &
$-4$    & $5.62503\ldots$ \\
$-8$    & $5.35895\ldots$ &
$-5$    & $5.35895\ldots$ \\
$-7$    & $5.21836\ldots$ &
$-6$    & $5.21836\ldots$
\end{tabular}
\end{minipage}

\caption{Representative volumes of Mazur and Jester manifold boundaries.}
\label{tab:volumes}
\end{table}

Together with computations for all $\card{k}\leq1000$,
Table~\ref{tab:volumes} suggests the following conjectures.
The volumes of Mazur manifold boundaries $\partial M_k$
appear to strictly increase as $k$ increases beginning at $-1$,
and as $k$ decreases beginning at $-6$.
Therefore,
\[
\vol{\partial M_{-1}}=1.39850\ldots
\]
is the minimum volume among all hyperbolic Mazur manifold boundaries.

The Jester manifold boundaries exhibit an apparent symmetry:
\[
\vol{\partial N_k}=\vol{\partial N_{-k-13}}
\]
for all $k\in\mathbb Z$.
The companion computations verify this symmetry numerically for $\card{k}\leq1000$.
We will prove this symmetry in forthcoming work.
Despite this symmetry, Theorem~\ref{maintheorem} shows that
all Jester manifold boundaries are nevertheless pairwise nonhomeomorphic.

The volumes of Jester manifold boundaries $\partial N_k$
appear to strictly increase as $k$ increases beginning at $-6$,
and as $k$ decreases beginning at $-7$.
Therefore,
\[
\vol{\partial N_{-6}}=\vol{\partial N_{-7}}=5.21836\ldots
\]
are the minimum-volume Jester manifold boundaries.
Since every hyperbolic Mazur manifold boundary satisfies
\[
\vol{\partial M_k} < \vol{X}=3.66386\ldots,
\]
strict volume monotonicity will provide an
alternative proof that no Jester manifold boundary is homeomorphic
to any Mazur manifold boundary.

Using Matveev and Polyak~\cite[Thm.~6.3]{mp}, the Casson invariant of every
Mazur manifold boundary is $-2$ and the Casson invariant of every
Jester manifold boundary is $0$.
Thus, the Casson invariant distinguishes Mazur manifold boundaries
from Jester manifold boundaries.

By Akbulut and Karakurt~\cite[Prop.~1.2]{aka} and
Saveliev~\cite{saveliev}, Heegaard Floer homology and instanton
Floer homology are independent of the framing within both the
Mazur and Jester families.
Thus, these Floer-theoretic invariants do not distinguish Mazur
manifold boundaries from one another and do not distinguish Jester
manifold boundaries from one another.

After completion of this work, Lidman brought to our attention
the work of Daemi, Eismeier, and Lidman~\cite[Thm.~1.8]{lidman}.
They show that the Chern--Simons filtration on instanton Floer homology
distinguishes $3$-manifolds within the Mazur and Jester families
up to orientation-preserving homeomorphism.
Since Mazur manifold boundaries have Casson invariant $-2$,
they admit no orientation-reversing homeomorphisms.
Thus, their work also distinguishes Mazur manifold boundaries
up to homeomorphism, regardless of orientation.
For Jester manifold boundaries, however, the Casson invariant vanishes,
so this argument does not extend to the unoriented classification.

In contrast, the geometric framework developed here connects all of these
phenomena to a single object: the regular ideal octahedron.
It provides a unified perspective on Mazur's knot, Jester's knot,
their Dehn fillings, and the resulting contractible $4$-manifolds.

\bibliographystyle{alpha}
\bibliography{References.bib}

\end{document}